\documentclass[12pt,a4paper]{amsart}

\usepackage{amssymb}
\usepackage[final]{showkeys}
\usepackage{microtype}
\usepackage{color}
\usepackage{graphicx}
\usepackage[hmargin=3cm,vmargin={5cm,4cm}]{geometry}

\newtheorem{te}{Theorem}[section]

\newtheorem{os}[te]{Remark}

\numberwithin{equation}{section}

\allowdisplaybreaks

\begin{document}

\title[Cyclic planar random motions with four direcions]{Cyclic random motions with orthogonal directions}
\author{E. Orsingher$^1$}
\address{$^1$ Dipartimento di Scienze Statistiche, Sapienza University of Rome}
\author{R. Garra$^1$}
\address{$^1$ Dipartimento di Scienze Statistiche, Sapienza University of Rome}
\author{A. I. Zeifman$^2$ $^3$ $^4$}
\address{$^2$ Vologda State University, Russia}
\address{$^3$ Institute of Informatics Problems, Federal Research Center “Informatics and Control” of RAS, Russia}
\address{$^4$ Vologda Research Center of RAS, Russia}

\begin{abstract}
A cyclic random motion at finite velocity with orthogonal directions is considered in the plane and in $\mathbb{R}^3$.
We obtain in both cases the explicit conditional distributions of the position of the moving particle
when the number of switches of directions is fixed. The explicit unconditional distributions
are also obtained and are expressed in terms of Bessel functions. 
The governing equations are derived and given as products of D'Alembert operators.
The limiting form of the equations is provided in the Euclidean space $\mathbb{R}^d$ and
takes the form of a heat equation with infinitesimal variance $1/d$.

\bigskip


\end{abstract}

\maketitle

\section{Introduction}

Random motions at finite velocity with a finite number of directions have been introduced
and analyzed since the eighties by different authors (see e.g. \cite{kol,kol1,enzo} and
the references therein).
We distinguish these motions according to the number of possible directions, the chance mechanism governing
the change of directions and the processes adopted to govern these changes. Of course, the space dimension
is important and throughout this paper we suppose that the motions considered develop in 
Euclidean spaces $\mathbb{R}^d$, with $d\geq 2$, although extensions of motions in non-Euclidean 
spaces have also been studied in the literature (see for example \cite{vale}).
The case of minimal cyclic random motions (see \cite{Lachal06} and \cite{Lachal})
has been studied in full generality by adopting the order statistics technique.
The interesting fact about these motions is that the distributions are expressed in terms
of hyper-Bessel functions, the prototype of which has the form 
\begin{equation}
I_{0,n}(x) =\sum_{k=0}^\infty \left(\frac{x}{n}\right)^{nk}\frac{1}{k!^n}.
\end{equation}
However, the probability distributions obtained are rather cumbersome even in the planar case.
In the case of orthogonal directions in the plane, the situation seems a little bit better and in some
cases the explicit distributions for the position of the randomly moving particle are obtained.
This is the case where the particle choses initially with probability 1/4 one of the four 
orthogonal directions with two possible changes of direction at Poisson times (with probability 1/2 the new 
directions are those orthogonal to that before the switch).
In this case, the current position of motion $\left(X(t),Y(t)\right)$ can be expressed in terms of independent 
telegraph processes $U(t)$ and $V(t)$ as
\begin{equation}\label{i0}
\begin{cases}
X(t)= U(t)+V(t),\\
Y(t) = U(t)-V(t).
\end{cases}
\end{equation}
If the planar motion has velocity $c$ and the changes of direction are governed by a homogeneous Poisson 
process with rate $\lambda$, then $U(t)$ and $V(t)$ are independent telegraph processes
with parameters $\left(\frac{c}{2},\frac{\lambda}{2}\right)$.
For details on this type of motion, see \cite{enzo}.
In this case, the technique adopted is based on the analysis of the fourth-order equation 
governing the distribution $p(x,y,t)$ of \eqref{i0} which reads 
\begin{equation}
\left(\frac{\partial}{\partial t} +\lambda\right)^2\bigg[\frac{\partial^2}{\partial t^2}-c^2 \Delta+2\lambda\frac{\partial}{\partial t}\bigg]p+c^4 \frac{\partial^4 p}{\partial x^2 \partial y^2} = 0.
\end{equation} 
Unfortunately, the order of the equations is equal to the number of possible
directions of motions which is a crucial difficulty in the derivation 
of the probability distributions.
In the present paper we study cyclic motions at finite velocity with 
orthogonal directions in $\mathbb{R}^2$ and $\mathbb{R}^3$ and are able 
to obtain explicit distributions of the vector processes
$\left(X(t), Y(t)\right)$ and $\left(X(t), Y(t), Z(t)\right)$.
From the analysis presented it emerges that in orthogonal motions in $\mathbb{R}^d$
performed at velocity $c$ and with a cyclic regime, the vector process
$\left(X_1(t),\dots, X_d(t)\right)$ has distribution $p(x_1, \dots, x_d, t)$
satisfying the equation of order $2d$ (that is the number of possible directions)
\begin{equation}\label{i1}
\prod_{i=1}^d \bigg[\left(\frac{\partial}{\partial t}+\lambda\right)^2-c^2\frac{\partial^2}{\partial x_i^2}\bigg]p-
\lambda^{2d}p = 0.
\end{equation}
At time $t$ the set of possible positions is a solid $S^d_{ct}$ with $2d$ vertices.
The outer shell of this set is reached in the cyclic case only if the number 
of Poisson events is less or equal to $d$.
Starting from the $d+1$-th Poisson event the inner points of $S^d_{ct}$ are reached with 
an increasing probability as time $t$ flows. The distribution $p(x_1, \dots, x_d, t)$
is uniform on the set homothetic to the outer surface $\partial S^d_{ct}$. 
For example, in the plane $S^2_{ct}$, this is a square with four vertices and the absolutely
continuous component of the distribution is uniform on squares
\begin{equation}
z = \pm x\pm y, \quad 0\leq z\leq ct.
\end{equation}
This is tantamount to reducing equation \eqref{i1} to the differential system 
\begin{equation}\label{i4}
\bigg[\left(\frac{\partial}{\partial t}+\lambda\right)^2-c^2\frac{\partial^2}{\partial x_k^2}\bigg] p = \lambda^2 p, \quad k = 1, 2, \quad d.
\end{equation}
which on $S^d_z$ becomes
\begin{equation}
\bigg[\left(\frac{\partial}{\partial t}+\lambda\right)^2-c^2\frac{\partial^2}{\partial z^2}\bigg] p = \lambda^2 p
\end{equation}
and coincides with a Klein-Gordon equation after the exponential transformation
$p = e^{-\lambda t}w $. This permits us to write down 
the distribution $p(z,t)$ also in the $d$-dimensional case in the form
\begin{equation}\label{i2}
p(z,t) = e^{-\lambda t}\sum_{i=0}^d A_i\frac{\partial^i}{\partial t^i}I_0\left(\frac{\lambda}{c}\sqrt{c^2t^2-z^2}\right), \quad 0\leq z \leq t,
\end{equation}
where the real coefficients $A_i$ appearing in \eqref{i2} are obtained by imposing the condition
\begin{equation}
\int_0^{ct} p(z,t)dz = 1-\sum_{k=0}^{d-1} P\left(N(t) = k \right),
\end{equation}
where $N(t)$ denotes the homogeneous Poisson process.
In the next sections we derive the equations \eqref{i4} explicitly in $\mathbb{R}^2$
and $\mathbb{R}^3$ and give the distributions \eqref{i2} with the explicit 
values of the coefficients $A_i$.\\
We are able to extract from our analysis the explicit conditional distributions which in $\mathbb{R}^3$
take the form 
\begin{equation}\nonumber
\begin{cases}
& \Pr\{\mathcal{U}(t)\in du| N(t) = 2k+2\} =\frac{du \ (2k+2)!}{(k+2)!(k-1)!}\frac{(c^2t^2-u^2)^{k-1}}{(2ct)^{2k+1}}(c^2t^2+3u^2), \ k\geq 1 \\
& \Pr\{\mathcal{U}(t)\in du| N(t) = 2k+1\} =\frac{du \ (2k+1)!}{2^{2k}(ct)^{2k+1}(k-1)!(k+1)!}(c^2t^2-u^2)^{k-1}(c^2t^2+u^2), \ k\geq 1,
\end{cases}
\end{equation}

where $\mathcal{U}(t)$ represents the relevant parameter of the octahedron of uniform distribution.
For $N(t)= 0$ the random particle reaches the vertices of $S^3_{ct}$.

\section{Cyclic planar random motion with orthogonal directions}

In this section we study a planar random motion with orthogonal directions taken cyclically at Poisson paced times. 
This model is suitable for describing the motion of a particle in a vortex. We assume that the possible
directions of motion $d_j$, $j = 1,2,3,4$ are orthogonal and represented by the vectors
\begin{align}
\nonumber & d_1 = (1,0), \quad d_2 =(0,1)\\
\nonumber & d_3 = (-1,0),\quad d_4 =(0,-1),
\end{align}
or in a more compact form
\begin{equation}
\nonumber d_k = \left(\cos \frac{k\pi}{2}, \sin\frac{k\pi}{2}\right), \quad k = j-1.
\end{equation}
The particle takes the directions cyclically, that is from $d_j$ it switches to $d_{j+1}$ and,
of course, $d_j = d_{j+4}$. The initial direction is chosen by a uniform law.
The sample paths of this motion are therefore of the form depicted in Fig.1. \\

   \begin{figure}
            \centering
            \includegraphics[scale=.73]{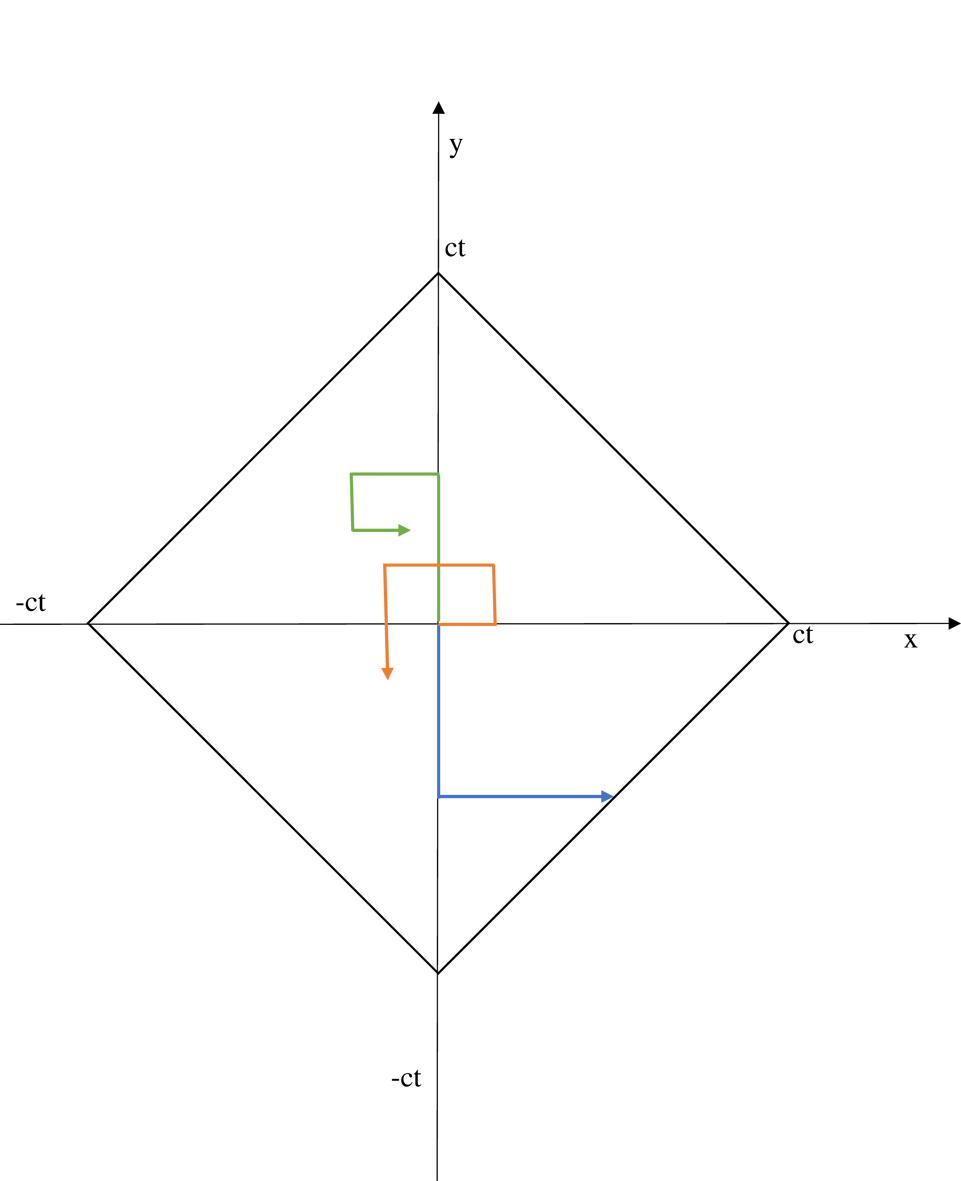}
            \caption{Some sample paths of the planar random motion with different 
            initial directions are depicted.}
            \label{figura}
        \end{figure}

In order to analyze this random motion we need the following four probability distributions
\begin{equation}
f_j(x,y,t) \ dx \ dy = \Pr\{X(t)\in dx, Y(t)\in dy, D(t)= d_j\}, \quad j = 1,2,3,4
\end{equation}
The process $D(t)$, $t>0$, takes values $d_j$, $j= 1,2,3,4$ and represents the current direction
of motion.

By $\lambda$ we denote the rate of the homogeneous Poisson process governing the changes of direction.
The position of the moving particle at time $t$ is in the square
\begin{equation}
S_{ct}= \{(x,y): |x+y|\leq ct, \ |x-y|\leq ct\}.
\end{equation}
The border $\partial S_{ct}$ is reached by the particle when either no Poisson event occurs or
exactly one Poisson event deviates the motion, i.e.
\begin{equation}
\Pr\{(X(t),Y(t))\in \partial S_{ct}\} = \Pr(N(t)= 0)+\Pr(N(t)= 1) = e^{-\lambda t}+\lambda t e^{-\lambda t}
\end{equation}

The absolutely continuous component of the probability distribution 
\begin{equation}\label{pd}
P(x,y,t)= \Pr\{X(t)\leq x, Y(t)\leq y \},
\end{equation}
has density given by
\begin{equation}
p(x,y,t) = f_1+f_2+f_3+f_4.
\end{equation}
Therefore, we are firstly interested in obtaining a single PDE governing $p(x,y,t)$ starting from the
full set of the equations governing the distributional laws of this kind of cyclic planar random motion.

The probabilities $f_j(x,y,t)$, $j = 1,2,3,4$ satisfy the following differential system
\begin{equation}\label{syst}
\begin{cases}
\frac{\partial f_1}{\partial t} = -c \frac{\partial f_1}{\partial x}+\lambda (f_4-f_1)\\
\frac{\partial f_2}{\partial t} = -c \frac{\partial f_2}{\partial y}+\lambda (f_1-f_2)\\
\frac{\partial f_3}{\partial t} = c \frac{\partial f_3}{\partial x}+\lambda (f_2-f_3)\\
\frac{\partial f_4}{\partial t} = c \frac{\partial f_4}{\partial y}+\lambda (f_3-f_4)
\end{cases}
\end{equation}
Therefore, we have the following result

\begin{te}
The absolutely continuous component of the probability density of the distribution \eqref{pd} satisfies 
the fourth order partial differential equation 
\begin{equation}\label{1}
\bigg[\left(\frac{\partial}{\partial t} +\lambda\right)^2-c^2\frac{\partial^2}{\partial x^2} \bigg]\bigg[\left(\frac{\partial}{\partial t} +\lambda\right)^2-c^2\frac{\partial^2}{\partial y^2}\bigg]p-\lambda^4 p = 0 
\end{equation} 
\end{te}

We observe that equation \eqref{1} can alternatively be rewritten as follows 
\begin{equation}
\left(\frac{\partial}{\partial t} +\lambda\right)^2\bigg[\frac{\partial^2}{\partial t^2}+
\lambda^2-c^2 \Delta+2\lambda\frac{\partial}{\partial t}\bigg]p+c^4 \frac{\partial^4 p}{\partial x^2 \partial y^2} -\lambda^4 p = 0
\end{equation} 

\begin{os}
Observe that for the projection of the motion on the $x$-axis, the governing equation reads
\begin{equation}
\left(\frac{\partial}{\partial t} +\lambda\right)^2\bigg[\left(\frac{\partial}{\partial t} +\lambda\right)^2-c^2\frac{\partial^2}{\partial x^2} \bigg]p -\lambda^4 p = 0,
\end{equation}
see equation (4.7) of \cite{Leo}. The analysis of the motion projected on the $x$-axis was carried out in this paper 
by using the order statistics technique. In this case, the motion on the line is similar to the 
telegraph process with stops between successive displacements. 

For the planar motion with four orthogonal directions (not cyclic) studied in \cite{enzo}
the equation governing the absolutely continuous component of the distribution has a similar
form, that is 
\begin{equation}
\left(\frac{\partial}{\partial t} +\lambda\right)^2\bigg[\frac{\partial^2}{\partial t^2}-c^2 \Delta+2\lambda\frac{\partial}{\partial t}\bigg]p+c^4 \frac{\partial^4 p}{\partial x^2 \partial y^2} = 0
\end{equation}  
\end{os}

Moreover, if we consider the functions $w_k = f_k e^{-\lambda t}$, $1\leq k \leq 4$, the system \eqref{syst} becomes
\begin{equation}\label{syst1}
\begin{cases}
\frac{\partial w_1}{\partial t} = -c \frac{\partial w_1}{\partial x}+\lambda w_4\\
\frac{\partial w_2}{\partial t} = -c \frac{\partial w_2}{\partial y}+\lambda w_1\\
\frac{\partial w_3}{\partial t} = c \frac{\partial w_3}{\partial x}+\lambda w_2\\
\frac{\partial w_4}{\partial t} = c \frac{\partial w_4}{\partial y}+\lambda w_3
\end{cases}
\end{equation}
and leads to a more simple equation for the function $w(x,y,t) = w_1+w_2+w_3+w_4$ that is
\begin{equation}\label{w}
\left(\frac{\partial^2}{\partial t^2}-c^2\frac{\partial^2}{\partial x^2}\right)\left(\frac{\partial^2}{\partial t^2}-c^2\frac{\partial^2}{\partial y^2}\right)w-\lambda^4 w = 0.
\end{equation}

The equation \eqref{w} can be split into the differential system of two Klein-Gordon-type 
equations
\begin{equation}\label{sys}
\begin{cases}
&\displaystyle \left(\frac{\partial^2}{\partial t^2}-c^2\frac{\partial^2}{\partial x^2}\right) w = \lambda^2 w, \\
& \displaystyle \left(\frac{\partial^2}{\partial t^2}-c^2\frac{\partial^2}{\partial y^2}\right)w = \lambda^2 w.
\end{cases}
\end{equation}
The Fourier transform of $w(x,y,t)$ with respect to $x$ and $y$, namely $\widehat{w}(\alpha,\beta,t)$ therefore satisfies the fourth order equation
\begin{equation}
\frac{d^4 \widehat{w}}{dt^4}-c^2[\alpha^2+\beta^2]\frac{d^2 \widehat{w}}{dt^2}+[c^4\alpha^2\beta^2-\lambda^4]\widehat{w} = 0
\end{equation}

\bigskip

We now examine the conditional characteristic function 
\begin{align}
G_n^{d_j}(\alpha, \beta, t)&=\mathbb{E}\left(e^{i\alpha X(t)+i\beta Y(t)}|N(t)=n, D(0) = d_j\right)=\\ 
\nonumber & = \frac{n!}{t^n}\int_0^t ds_1 \int_{s_1}^t ds_2\dots \int_{s_{n-1}}^t ds_n
\prod_{k=1}^{n+1}e^{ic(s_k-s_{k-1})(\alpha \sin(\frac{k\pi}{2}-(j-1)\frac{\pi}{2})+\beta \cos(\frac{k\pi}{2}-(j-1)\frac{\pi}{2})}\\
\nonumber & = \frac{n!}{t^n}\int_0^t ds_1 \int_{s_1}^t ds_2\dots \int_{s_{n-1}}^t ds_n
\prod_{k=1}^{n+1}e^{ic(s_k-s_{k-1})(\alpha \cos((k-j)\frac{\pi}{2})-\beta \sin((k-j)\frac{\pi}{2}))}\\
\nonumber &=\frac{n!}{t^n}F_n^{d_j}(\alpha, \beta, t), \quad j = 1, 2, 3, 4,
\end{align}
where $0= s_0< s_1<\dots < s_{k+1} = t$.
We have the following result

\begin{te}
The functions $F_n^{d_j}(\alpha, \beta, t)= \frac{t^n}{n!}G_n^{d_j}(\alpha, \beta, t)$ satisfy
the first-order difference-differential equations
\begin{align}
\frac{dF_n^{d_j}}{dt}&= F_{n-1}^{d_j}+ic\{\alpha \cos\{(n+1-j)\frac{\pi}{2}\}-\beta \sin\{(n+1-j)\frac{\pi}{2}\}\}F_n^{d_j}.
\end{align}
The functions $F_n(\alpha, \beta, t) = \frac{1}{4} \sum_{j=1}^4 F_n^{d_j}(\alpha, \beta, t)$
satisfy instead
\begin{equation}
\frac{dF_n}{dt}= F_{n-1}+\frac{ic}{4}\{\alpha\sum_{j=1}^4 \cos\{(n+1-j)\frac{\pi}{2}\} F_n^{d_j}-\beta \sum_{j=1}^4\sin\{(n+1-j)\frac{\pi}{2}\}F_n^{d_j}\}.
\end{equation}
\end{te}

\section{The explicit distribution}

We first observe that, for
\begin{equation}\label{change}
(x,y): 
\begin{cases}
x+y =\pm u, \quad |x-y|<u,\\
x-y = \pm u, \quad |x+y|<u,
\end{cases}
\end{equation}
the distribution \eqref{pd} is uniform on the square layer $S_{u}$ depicted in Fig.2 which is
homothetic with the support square $S_{ct}$.
  \begin{figure}
            \centering
            \includegraphics[scale=.73]{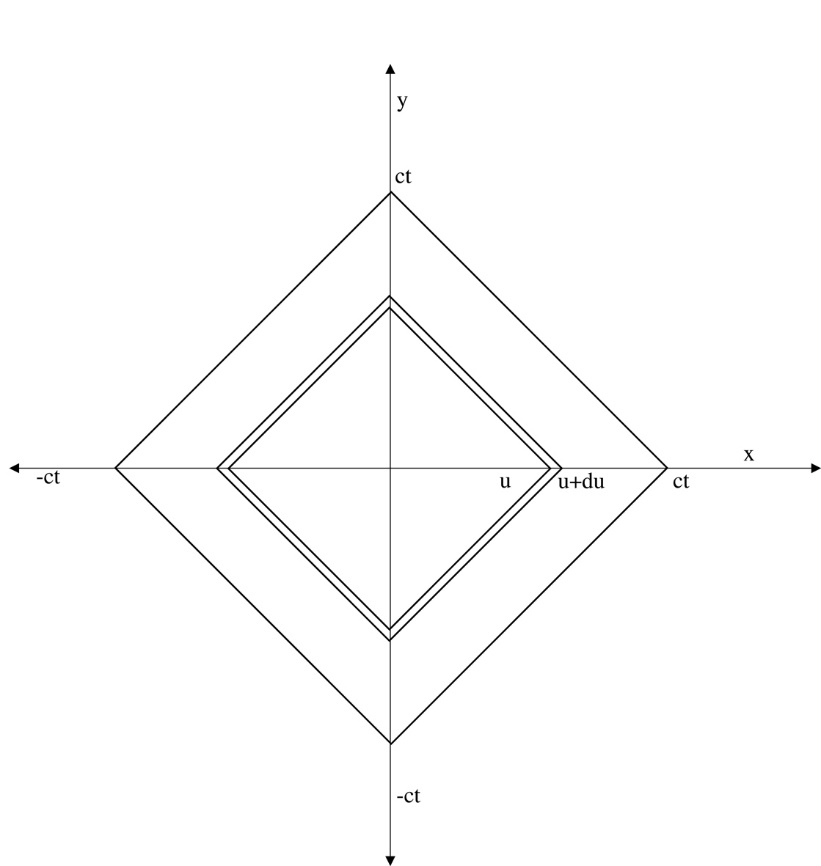}
            \caption{The strips $S_u$ omothetic to the support set $S_{ct}$ are depicted. We denoted by $u$ the half-diagonal of $S_u$. The process $\mathcal{U}(t)$ takes the values $u$ and has distribution $p(u,t)$ given by \eqref{dist}.
           .}
            \label{figura2}
        \end{figure}
By means of the change of variable $u = \pm x\pm y$, we reduce equation \eqref{w} to the differential system
of Klein-Gordon-type equations \eqref{sys}.
Note that the time-derivatives of the function $\displaystyle I_0\left(\frac{\lambda}{c}\sqrt{c^2t^2-u^2}\right)$
are also solutions of equations \eqref{sys} because of the homogeneity of 
\begin{equation}\label{star}
\frac{\partial^2 w}{\partial t^2}-c^2 \frac{\partial^2 w}{\partial u^2}-\lambda^2 w= 0.
\end{equation}
 Therefore, the solution to \eqref{w}, in view of \eqref{sys}, can be written in the form 
\begin{equation}\label{sol}
w(u,t)= A I_0\left(\frac{\lambda}{c}\sqrt{c^2t^2-u^2}\right)+B\frac{\partial}{\partial t}I_0\left(\frac{\lambda}{c}\sqrt{c^2t^2-u^2}\right)+
C \frac{\partial^2}{\partial t^2}I_0\left(\frac{\lambda}{c}\sqrt{c^2t^2-u^2}\right),
\end{equation}
where $A, B, C$ are real coefficients that will be defined in the following and 
\begin{equation}
I_0(t) = \sum_{k=0}^\infty \left(\frac{t}{2}\right)^{2k}\frac{1}{k!^2},
\end{equation}
is the modified Bessel function.\\
We have the following result

\begin{te}
The absolutely continuous component of the distribution \eqref{pd} has density
\begin{align}\label{dist}
p(u,t) &= \frac{e^{-\lambda t}}{c}\left[-\lambda I_0\left(\frac{\lambda}{c}\sqrt{c^2t^2-u^2}\right)+\frac{\partial}{\partial t}I_0\left(\frac{\lambda}{c}\sqrt{c^2t^2-u^2}\right)+\frac{2}{\lambda}\frac{\partial^2}{\partial t^2}I_0\left(\frac{\lambda}{c}\sqrt{c^2t^2-u^2}\right)\right]\\
\nonumber & =\frac{e^{-\lambda t}}{c}\left[\lambda I_0\left(\frac{\lambda}{c}\sqrt{c^2t^2-u^2}\right)+\frac{\partial}{\partial t}I_0\left(\frac{\lambda}{c}\sqrt{c^2t^2-u^2}\right)+\frac{2c^2}{\lambda}\frac{\partial^2}{\partial u^2}I_0\left(\frac{\lambda}{c}\sqrt{c^2t^2-u^2}\right)\right]  ,
\end{align} 
where $0<u<ct$.
\end{te}

\begin{proof}
It is straightforward that equations \eqref{sys} after the change of variable \eqref{change}
are reduced to the form \eqref{star} which is satisfied by \eqref{sol}. This is a linear combination of non-negative functions. Going back to 
$p(u,t) = e^{-\lambda t} w(u,t)$, we must now find the coefficients $A, B$ and $C$ such that 
\begin{equation}\label{co}
\int_{0}^{ct} p(u,t)du = 1-e^{-\lambda t}- \lambda t \cdot e^{-\lambda t},
\end{equation}
since we have that a singular part of the distribution is concentrated on the boundary and pertains to the probability that the particle has one or zero changes of direction.\\
By using the following results
\begin{align}
&\int_0^{ct} I_0\left(\frac{\lambda}{c}\sqrt{c^2t^2-u^2}\right) du = \frac{c}{2\lambda}(e^{\lambda t}-e^{-\lambda t}),\\
&\int_0^{ct}\frac{\partial}{\partial t}I_0\left(\frac{\lambda}{c}\sqrt{c^2t^2-u^2}\right) du = 
\frac{c}{2}(e^{\lambda t}+e^{-\lambda t})-c,\\
& \int_0^{ct} \frac{\partial^2}{\partial t^2}I_0\left(\frac{\lambda}{c}\sqrt{c^2t^2-u^2}\right) du = \frac{c\lambda}{2}(e^{\lambda t}-e^{-\lambda t})-\frac{c\lambda^2 t}{2}
\end{align}
under the constraint \eqref{co} we finally obtain the claimed result.

The second form of \eqref{dist} can be obtained by considering that
\begin{equation}
\nonumber \frac{\partial^2}{\partial t^2}I_0\left(\frac{\lambda}{c}\sqrt{c^2t^2-u^2}\right)=
c^2 \frac{\partial^2}{\partial u^2} I_0\left(\frac{\lambda}{c}\sqrt{c^2t^2-u^2}\right)+
\lambda^2 I_0\left(\frac{\lambda}{c}\sqrt{c^2t^2-u^2}\right).
\end{equation}

\end{proof}

\begin{os}
We can write the distribution \eqref{dist} in two alternative forms.
The first one is expressed in terms of the functions $I_0(\cdot)$ and $I_1(\cdot)$ as follows
\begin{align}
p(u,t) &= e^{-\lambda t}\bigg[\frac{\lambda}{c}\cdot\frac{c^2t^2+u^2}{c^2t^2-u^2} \cdot I_0\left(\frac{\lambda}{c}\sqrt{c^2t^2-u^2}\right)+\\
\nonumber &+\frac{1}{\sqrt{c^2t^2-u^2}}\bigg\{
\frac{\lambda t(c^2t^2-u^2)-2(u^2+c^2t^2)}{c^2t^2-u^2}\bigg\} I_1\left(\frac{\lambda}{c}\sqrt{c^2t^2-u^2}\right)\bigg].
\end{align}
A second useful representation of the distribution \eqref{dist} is given by 
\begin{equation}
p(u,t) = \frac{\lambda}{c} \frac{e^{-\lambda t}}{c^2t^2-u^2}\sum_{k=0}^{\infty}
\left(\frac{\lambda}{2c}\right)^{2k}\frac{(\sqrt{c^2t^2-u^2})^{2k}}{k!^2}\bigg[(c^2t^2+u^2)\left(1-\frac{1}{k+1}\right)+\frac{\lambda t(c^2t^2-u^2)}{2(k+1)}\bigg],
\end{equation}
from which the non-negativity of \eqref{dist} can be directly ascertained.
\end{os}

Let us denote with $\mathcal{U}(t)$ the random variable whose density law is represented by \eqref{dist}. We can now explicitly compute the mean value of $\mathcal{U}(t)$.
\begin{te}
We have the following result
\begin{equation}\label{mean}
\mathbb{E} \ \mathcal{U}(t) = e^{-\lambda t}\cdot\bigg[\left(ct+\frac{2c}{\lambda}\right)I_0(\lambda t)+ct I_1(\lambda t)-\frac{2c}{\lambda}\bigg]
\end{equation}
\end{te}
\begin{proof}
The result \eqref{mean} can be obtained by taking into account the following formulas
\begin{align}
\nonumber & \int_0^{ct} u I_0\left(\frac{\lambda}{c}\sqrt{c^2t^2-u^2}\right)du = \frac{c^2 t}{\lambda} I_1(\lambda t),\\
\nonumber & \int_0^{ct} u \frac{\partial}{\partial t}I_0\left(\frac{\lambda}{c}\sqrt{c^2t^2-u^2}\right) du = c^2 t I_0(\lambda t)-c^2 t,\\
\nonumber & \int_0^{ct} u\frac{\partial^2}{\partial t^2}I_0\left(\frac{\lambda}{c}\sqrt{c^2t^2-u^2}\right) du = c^2 I_0(\lambda t)+c^2 t\lambda
I_1(\lambda t)-c^2-\frac{\lambda^2 c^2 t^2}{2}
\end{align}
and by adding the mean value part related to the singular component of the distribution.

\end{proof}

Observe that for $t\rightarrow 0$ we have that
\begin{equation}
\mathbb{E} \ \mathcal{U}(t) \sim ct.
\end{equation}

We now evaluate the $m$-th moment of the r.v. $\mathcal{U}(t)$ with distribution \eqref{dist}.
We need the following results
\begin{align}
\int_0^{ct} u^m I_0\left(\frac{\lambda}{c}\sqrt{c^2t^2-u^2}\right)du &= 
\frac{1}{2}\Gamma\left(\frac{m+1}{2}\right)\left(\frac{2c^2t}{\lambda}\right)^{\frac{m+1}{2}}
I_{\frac{m+1}{2}}(\lambda t)\\
\nonumber\label{mth} \int_0^{ct} u^m \frac{\partial}{\partial t}I_0\left(\frac{\lambda}{c}\sqrt{c^2t^2-u^2}\right) du &=
\frac{1}{2t}\Gamma\left(\frac{m+1}{2}+1\right)\left(\frac{2c^2t}{\lambda}\right)^{\frac{m+1}{2}}
I_{\frac{m+1}{2}}(\lambda t)+\\
&+\frac{1}{2}\Gamma\left(\frac{m+1}{2}\right)\left(\frac{2c^2t}{\lambda}\right)^{\frac{m+1}{2}}\frac{\lambda}{2}\bigg[I_{\frac{m-1}{2}}(\lambda t)+I_{\frac{m+3}{2}}(\lambda t)\bigg]-c(ct)^m.
\end{align}
We simplify \eqref{mth} by means of the relationship
\begin{equation}
\nonumber I_{\frac{m+3}{2}}(\lambda t) = I_{\frac{m-1}{2}}(\lambda t)-\frac{2}{\lambda t}\frac{m+1}{2}I_{\frac{m+1}{2}}(\lambda t),
\end{equation}
which for $m = 2r-1$ yields the classical recurrence relationship for integer-order
Bessel functions
\begin{equation}
\nonumber I_{r+1}(\lambda t) = I_{r-1}(\lambda t)-\frac{2r}{\lambda t}I_r(\lambda t).
\end{equation}
Therefore we have that \eqref{mth} can be reduced to the following form
\begin{equation}
 \int_0^{ct} u^m \frac{\partial}{\partial t}I_0\left(\frac{\lambda}{c}\sqrt{c^2t^2-u^2}\right) du =
 \frac{\lambda}{2}\left(\frac{2c^2t}{\lambda}\right)^{\frac{m+1}{2}}\Gamma\left(\frac{m+1}{2}\right)I_{\frac{m-1}{2}}(\lambda t)-c(ct)^m.
\end{equation}
In order to evaluate the integral
\begin{equation}
\nonumber \int_0^{ct} u^m \frac{\partial^2}{\partial t^2}I_0\left(\frac{\lambda}{c}\sqrt{c^2t^2-u^2}\right) du
\end{equation}
we need again the following equality
\begin{equation}\label{kgg}
\frac{\partial^2}{\partial t^2}I_0\left(\frac{\lambda}{c}\sqrt{c^2t^2-u^2}\right) = c^2
\frac{\partial^2}{\partial u^2}I_0\left(\frac{\lambda}{c}\sqrt{c^2t^2-u^2}\right) +\lambda^2 I_0\left(\frac{\lambda}{c}\sqrt{c^2t^2-u^2}\right).
\end{equation}
Since
\begin{align}\nonumber 
&c^2 \nonumber \int_0^{ct} u^m \frac{\partial^2}{\partial u^2}I_0\left(\frac{\lambda}{c}\sqrt{c^2t^2-u^2}\right) du = -\frac{\lambda^2ct}{2}
(ct)^m-mc^2(ct)^{m-1}+\\
\nonumber &+m\Gamma\left(\frac{m+1}{2}\right)c^2 \left(\frac{2c^2t}{\lambda}\right)^{\frac{m-1}{2}}I_{\frac{m-1}{2}}(\lambda t),
\end{align}
we have that
\begin{align}\nonumber
& \int_0^{ct} u^m \frac{\partial^2}{\partial t^2}I_0\left(\frac{\lambda}{c}\sqrt{c^2t^2-u^2}\right) du = -\frac{\lambda^2ct}{2}
 (ct)^m-mc^2(ct)^{m-1}+\\
 \nonumber &+m\Gamma\left(\frac{m+1}{2}\right)c^2 \left(\frac{2c^2t}{\lambda}\right)^{\frac{m-1}{2}}I_{\frac{m-1}{2}}(\lambda t)+ \frac{\lambda^2}{2}\Gamma\left(\frac{m+1}{2}\right)\left(\frac{2c^2t}{\lambda}\right)^{\frac{m+1}{2}}
 I_{\frac{m+1}{2}}(\lambda t)
\end{align}
At this point we have that the $m$-th moment of the absolutely continuous part of the distribution
of $\mathcal{U}$, namely $f(u)$, is given by
\begin{align}
\nonumber & \int_0^{ct} u^m f(u)du = \frac{e^{-\lambda t}}{c}\bigg[\frac{\lambda}{2}\Gamma\left(\frac{m+1}{2}\right)\left(\frac{2c^2t}{\lambda}\right)^{\frac{m+1}{2}}
I_{\frac{m+1}{2}}(\lambda t)-c(ct)^m-\\
\nonumber & -\lambda (ct)^{m+1}-\frac{2}{\lambda}mc^2(ct)^{m-1}+ \Gamma\left(\frac{m+1}{2}\right) I_{\frac{m-1}{2}}(\lambda t)\left(\frac{\lambda}{2}\left(\frac{2c^2t}{\lambda}\right)^{\frac{m+1}{2}}+\frac{2mc^2}{\lambda}\left(\frac{2c^2t}{\lambda}\right)^{\frac{m-1}{2}}\right)
\bigg].
\end{align}
The contribution of the singular part of the distribution equals
\begin{equation}
(ct)^m e^{-\lambda t} (1+\lambda t).
\end{equation}
Taking all these results together, we are finally able to write the $m$-th moment of the r.v. $\mathcal{U}(t)$, that is given by 
\begin{align}
\nonumber &\mathbb{E} \ \mathcal{U}^m (t) =\frac{e^{-\lambda t}}{c}\bigg[\frac{\lambda}{2}\Gamma\left(\frac{m+1}{2}\right)\left(\frac{2c^2t}{\lambda}\right)^{\frac{m+1}{2}}
I_{\frac{m+1}{2}}(\lambda t)-\frac{2}{\lambda}mc^2(ct)^{m-1}+\\
\nonumber & +\Gamma\left(\frac{m+1}{2}\right) I_{\frac{m-1}{2}}(\lambda t)\left(\frac{\lambda}{2}\left(\frac{2c^2t}{\lambda}\right)^{\frac{m+1}{2}}+\frac{2mc^2}{\lambda}\left(\frac{2c^2t}{\lambda}\right)^{\frac{m-1}{2}}\right)
\bigg].
\end{align}
The reader can check that for $m = 1$, we get \eqref{mean}.

A useful result is given by the explicit conditional distributions of $\mathcal{U}(t)$, $t>0$,
which describe the behaviour of motion when the number $N(t)$ of the deviations is fixed. This
is particularly interesting for small values of $N(t)$.

\begin{te}
The conditional distributions of $\mathcal{U}(t)$, $t>0$, are 
\begin{equation}\label{3,1}
\Pr\{\mathcal{U}(t)\in du| N(t) = 2k+2\} =\frac{(2k+2)!}{k!(k+1)!}\frac{du}{(2ct)^{2k+1}}(c^2t^2-u^2)^k, \ k\geq 0 
\end{equation}
with $0\leq |u| \leq ct$ and 
\begin{equation}
\Pr\{\mathcal{U}(t)\in du| N(t) = 2k+1\} =\frac{(2k+1)!}{2^{2k}(ct)^{2k+1}du}\frac{(c^2t^2-u^2)^{k-1}}{(k-1)!(k+1)!}(c^2t^2+u^2), \ k\geq 1 
\end{equation}
\end{te}

\begin{proof}
It is convenient to write the distribution \eqref{dist} as follows 
\begin{equation}\label{con1}
p(u,t)= \frac{e^{-\lambda t}}{c}\left[\lambda I_0\left(\frac{\lambda}{c}\sqrt{c^2t^2-u^2}\right)+\frac{\partial}{\partial t}I_0\left(\frac{\lambda}{c}\sqrt{c^2t^2-u^2}\right)+\frac{2c^2}{\lambda}\frac{\partial^2}{\partial u^2}I_0\left(\frac{\lambda}{c}\sqrt{c^2t^2-u^2}\right)\right]
\end{equation}  
and examine each term separately.\\
The first term of \eqref{con1} can be conveniently rewritten as 
\begin{align}\label{con2}
&e^{-\lambda t} \sum_{k=0}^\infty \frac{(\lambda t)^{2k+1}}{(2k+1)!}\bigg\{\frac{(2k+1)!}{(ct)^{2k+1}2^{2k}}\frac{1}{k!^2}(c^2t^2-u^2)^k\bigg\}=\\
\nonumber & =\sum_{k=0}^\infty \Pr\{N(t)=2k+1\}\Pr\{\mathcal{U}(t)\in du| N(t) = 2k+1\}.
\end{align}
The second term can be rewritten as
\begin{align}
& \frac{e^{-\lambda t}}{c}\frac{\partial}{\partial t} \sum_{k=0}^\infty \left(\frac{\lambda}{2c}\right)^{2k}\frac{1}{k!^2}(c^2t^2-u^2)^k = \\
\nonumber &=\frac{e^{-\lambda t}}{c}\sum_{k=0}^\infty \left(\frac{\lambda}{2c}\right)^{2k}\frac{k}{k!^2}2c^2t(c^2t^2-u^2)^{k-1}=\\
\nonumber & =e^{-\lambda t}(2ct)\sum_{k=0}^\infty \left(\frac{\lambda}{2c}\right)^{2k+2}\frac{1}{k!(k+1)!}(c^2t^2-u^2)^{k}=\\
\nonumber & =e^{-\lambda t}\sum_{k=0}^\infty \frac{(\lambda t)^{2k+2}}{(2k+2)!}
\frac{(2k+2)!}{k!(k+1)!}\frac{1}{(2ct)^{2k+1}}(c^2t^2-u^2)^k=\\
\nonumber &= \sum_{k=0}^\infty  \Pr\{N(t)=2k+2\}\Pr\{\mathcal{U}(t)\in du| N(t) = 2k+2\}
\end{align}

Many more details are requested to analyze the third term which can be developed as follows
\begin{align}
\nonumber & e^{-\lambda t}\frac{2c}{\lambda}\frac{\partial^2}{\partial u^2}I_0\left(\frac{\lambda}{c}\sqrt{c^2t^2-u^2}\right)=\\
\nonumber &= - e^{-\lambda t}\frac{2^2c}{\lambda}\sum_{k=0}^\infty \left(\frac{\lambda}{2c}\right)^{2k}\frac{k}{k!^2}(c^2t^2-u^2)^{k-1}+\\
\nonumber & +e^{-\lambda t}\frac{2^3c}{\lambda}\sum_{k=0}^\infty \left(\frac{\lambda}{2c}\right)^{2k}\frac{k(k-1)}{k!^2}(c^2t^2-u^2)^{k-2}u^2=\\
\nonumber & = - e^{-\lambda t}2\sum_{k=1}^\infty \frac{(\lambda t)^{2k-1}}{(2k-1)!}\frac{(2k-1)!}{(k-1)!k!}\frac{(c^2t^2-u^2)^{k-1}}{(2ct)^{2k-1}}+\\
\nonumber & +e^{-\lambda t}2^2\sum_{k=2}^\infty \frac{(\lambda t)^{2k-1}}{(2k-1)!}\frac{(2k-1)!}{(k-2)!k!}\frac{u^2(c^2t^2-u^2)^{k-2}}{(2ct)^{2k-1}}=\\
\nonumber & = S_1+S_2.
\end{align}
We observe that the term $k= 0$ of \eqref{con2} and the first term of $S_1$ cancel each other out so that, putting together \eqref{con2} with $S_1$ and $S_2$, we have
\begin{align}
\nonumber & \sum_{k=1}^\infty \Pr\{N(t)=2k+1\}\frac{(2k+1)!}{(ct)^{2k+1}}\frac{1}{2^{2k}}\frac{1}{k!^2}(c^2t^2-u^2)^k-\\
\nonumber & -2\sum_{k=2}^\infty \Pr\{N(t)=2k-1\}\frac{(2k-1)!}{(k-1)!k!}\frac{(c^2t^2-u^2)^{k-1}}{(2ct)^{2k-1}}+\\
\nonumber & +2^2 \sum_{k=2}^\infty \Pr\{N(t)=2k-1\}\frac{(2k-1)!}{(k-2)!k!}\frac{u^2(c^2t^2-u^2)^{k-2}}{(2ct)^{2k-1}}=\\
\nonumber & = \sum_{k=1}^\infty \Pr\{N(t)=2k+1\}\frac{(2k+1)!}{(ct)^{2k+1}}\frac{1}{2^{2k}}\frac{1}{k!^2}(c^2t^2-u^2)^k-\\
\nonumber & -2\sum_{k=1}^\infty \Pr\{N(t)=2k+1\}\frac{(2k+1)!}{(k+1)!k!}\frac{(c^2t^2-u^2)^{k}}{(2ct)^{2k+1}}+\\
\nonumber & +2^2 \sum_{k=1}^\infty \Pr\{N(t)=2k+1\}\frac{(2k+1)!}{(k-1)!(k+1)!}\frac{u^2(c^2t^2-u^2)^{k-1}}{(2ct)^{2k+1}}= \\
\nonumber &=  \sum_{k=1}^\infty \Pr\{N(t)=2k+1\}\frac{(2k+1)!}{2^{2k}(ct)^{2k+1}}\frac{(c^2t^2-u^2)^{k-1}}{(k-1)!(k+1)!}(c^2t^2+u^2)=\\
\nonumber &= \sum_{k=1}^\infty \Pr\{N(t)=2k+1\}\Pr\{\mathcal{U}(t)\in du| N(t) = 2k+1\}
\end{align}
and this concludes the proof.
\end{proof}

\begin{os}
From \eqref{3,1} we infer that for $k=0$ we have that 
\begin{equation}
\Pr\{\mathcal{U}(t)\in du|N(t)= 2 \} = \frac{du}{ct}, \quad 0<u<t.
\end{equation}
Since on the squares $u = \pm x\pm y$ the distribution of $(X(t),Y(t))$ is uniform and
is also uniform in $0\leq u \leq ct $, we conclude that for two changes of direction
the moving particle is uniformly distributed inside the square $S_{ct}$.
This astonishing fact is also true for planar motions with an infinite number of
directions where we have (see \cite{ale})
\begin{equation}
\Pr\{X(t)\in dx, Y(t)\in dy|N(t)= 2\} = \frac{dxdy}{\pi(ct)^2}, \quad x^2+y^2<c^2t^2.
\end{equation}
For 
\begin{equation}
\Pr\{\mathcal{U}(t)\in du|N(t)= 3\} = \frac{3}{4}\frac{c^2t^2+u^2}{(ct)^3}du, \quad 0<u<ct,
\end{equation}
the maximal values of the distribution are attained near the border of $S_{ct}$.\\
For a number of changes of direction $N(t)\geq 4$ the maximal concentration of the distribution
of $\mathcal{U}(t)$ is near the starting point because the sample paths coil up around
the origin.
\end{os}

\section{Cyclic random motions in higher dimensions}

We can see that a similar analytical treatment can be developed also in higher order dimensions, considering cyclic random motions with orthogonal directions.

For example, in the interesting three dimensional case, it is possible to prove (after the exponential change of variable $p(x,y,z,t) = e^{-\lambda t}w(x,y,z,t)$) that the governing equation for the density of the absolutely continuous component of the probability distribution function
\begin{equation}\label{3d}
P(x,y,z,t)= \Pr\{X(t)\leq x, Y(t)\leq y, Z(t)\leq z \},
\end{equation}

is given by 
\begin{equation}\label{w3}
\left(\frac{\partial^2}{\partial t^2}-c^2\frac{\partial^2}{\partial x^2}\right)\left(\frac{\partial^2}{\partial t^2}-c^2\frac{\partial^2}{\partial y^2}\right)\left(\frac{\partial^2}{\partial t^2}-c^2\frac{\partial^2}{\partial z^2}\right)w-\lambda^6 w = 0.
\end{equation} 

If the function $w\equiv w(x,y,z,t)$ is a solution of the differential system 
\begin{equation}\label{sys3}
\begin{cases}
&\displaystyle \left(\frac{\partial^2}{\partial t^2}-c^2\frac{\partial^2}{\partial x^2}\right) w = \lambda^2 w, \\
& \displaystyle \left(\frac{\partial^2}{\partial t^2}-c^2\frac{\partial^2}{\partial y^2}\right)w = \lambda^2 w, \\
& \displaystyle \left(\frac{\partial^2}{\partial t^2}-c^2\frac{\partial^2}{\partial z^2}\right) w = \lambda^2 w,
\end{cases}
\end{equation}
then it solves the sixth-order partial differential equation \eqref{w3}.
On the surfaces $u = x\pm y\pm z$ we have the Klein-Gordon equation 
\begin{equation}
\left(\frac{\partial^2}{\partial t^2}-c^2\frac{\partial^2}{\partial u^2}\right) w = \lambda^2 w
\end{equation}
and, in analogy with what happens in the two-dimensional case we have that for the half-diagonal of
the octahedrons $S^3_{u}$, the distribution of the r.v. $\mathcal{U}(t)$ has an absolutely continuous component given in the next theorem.

\begin{te}
The absolutely continuous component of the distribution \eqref{3d} has density
\begin{align}\label{dist3}
p(u,t) =& \frac{e^{-\lambda t}}{c}\bigg[-\lambda I_0\left(\frac{\lambda}{c}\sqrt{c^2t^2-u^2}\right)-3\frac{\partial}{\partial t}I_0\left(\frac{\lambda}{c}\sqrt{c^2t^2-u^2}\right)+\frac{2}{\lambda}\frac{\partial^2}{\partial t^2}I_0\left(\frac{\lambda}{c}\sqrt{c^2t^2-u^2}\right)+\\
\nonumber &+\frac{4}{\lambda^2}\frac{\partial^3}{\partial t^3}I_0\left(\frac{\lambda}{c}\sqrt{c^2t^2-u^2}\right)\bigg]\\
\nonumber & = \frac{e^{-\lambda t}}{c}\bigg[\lambda I_0\left(\frac{\lambda}{c}\sqrt{c^2t^2-u^2}\right)+\frac{\partial}{\partial t}I_0\left(\frac{\lambda}{c}\sqrt{c^2t^2-u^2}\right)+\frac{2c^2}{\lambda}\frac{\partial^2}{\partial u^2}I_0\left(\frac{\lambda}{c}\sqrt{c^2t^2-u^2}\right)+\\
\nonumber &+\frac{4c^2}{\lambda^2}\frac{\partial^3}{\partial t\partial u^2}I_0\left(\frac{\lambda}{c}\sqrt{c^2t^2-u^2}\right)\bigg],
\end{align} 
where $0<u<ct$.
\end{te}

\begin{proof}

We first observe that equation \eqref{w3} can be reduced to a single equation in $(u,t)$ on the octahedron homothetic to $\partial S^3_{ct}$ (see Fig.3)
\begin{equation}\nonumber
\left(\frac{\partial^2}{\partial t^2}-c^2\frac{\partial^2}{\partial u^2}\right)w-\lambda^2 w = 0.
\end{equation}

and it admits solutions of the form 

\begin{align}
\nonumber w(u,t)&= A I_0\left(\frac{\lambda}{c}\sqrt{c^2t^2-u^2}\right)+B\frac{\partial}{\partial t}I_0\left(\frac{\lambda}{c}\sqrt{c^2t^2-u^2}\right)+C\frac{\partial^2}{\partial t^2}I_0\left(\frac{\lambda}{c}\sqrt{c^2t^2-u^2}\right)+\\
&\nonumber +D\frac{\partial^3}{\partial t^3}I_0\left(\frac{\lambda}{c}\sqrt{c^2t^2-u^2}\right),
\end{align}
where $A, B, C, D$ are real coefficients that must be chosen in such a way that the condition
\begin{equation}\label{co3}
\int_{0}^{ct} p(u,t)du = 1-e^{-\lambda t}-\lambda t \ e^{-\lambda t}-\frac{\lambda^2 t^2}{2} e^{-\lambda t},
\end{equation}
is fulfilled, since in this case, we have that the singular part of the distribution pertains to the probability that the particle has no change, one or two changes of directions.\\
By observing that
\begin{equation}
\int_0^{ct} \frac{\partial^3}{\partial t^3}I_0\left(\frac{\lambda}{c}\sqrt{c^2t^2-u^2}\right) du =
\frac{c\lambda^2}{2}(e^{\lambda t}+e^{-\lambda t})-\lambda^2 c-\frac{\lambda^4 c t^2}{8}
\end{equation}
we obtain the claimed result.\\

The second form of \eqref{dist3} is derived by applying \eqref{kgg} after some calculations.
\end{proof}

\begin{figure}
            \centering
            \includegraphics[scale=.63]{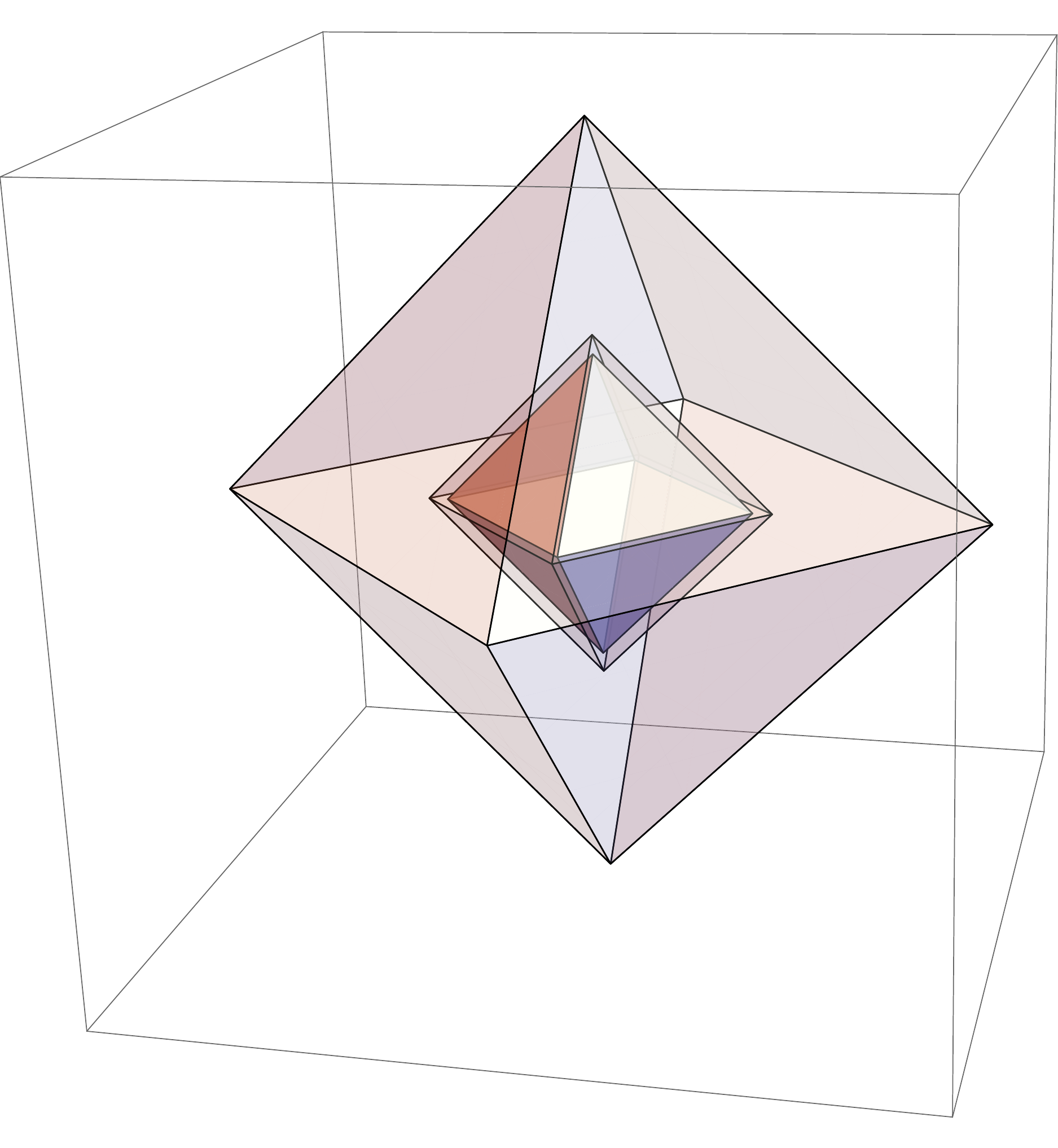}
            \caption{The octahedron $S_{ct}^3$ of the possible positions of 
            $(X(t),Y(t),Z(t))$ at time $t$ is represented. The surface of an octahedron
            homothetic to $S^3_{ct}$ is drawn. On $S^3_{ct}$ the distribitution
            of the moving particle is uniform.
            }
            \label{figura3}
        \end{figure}

\begin{te}
The conditional distributions for the cyclic motions in $\mathbb{R}^3$ are
\begin{equation}
\Pr\{\mathcal{U}(t)\in du| N(t) = 2k+2\} =\frac{du \ (2k+2)!}{(k+2)!(k-1)!}\frac{(c^2t^2-u^2)^{k-1}}{(2ct)^{2k+1}}(c^2t^2+3u^2), \ k\geq 1 
\end{equation}
and 
\begin{equation}
\Pr\{\mathcal{U}(t)\in du| N(t) = 2k+1\} =\frac{du \ (2k+1)!}{2^{2k}(ct)^{2k+1}(k-1)!(k+1)!}(c^2t^2-u^2)^{k-1}(c^2t^2+u^2), \ k\geq 1,
\end{equation}
with $0\leq u \leq ct$.
\end{te}

\begin{proof}
Let us denote with $p_2(u,t)$ the density of the two dimensional cyclic random motion and $p_3(u,t)$ the density in the three dimensional case.
The relationship
\begin{align}
\nonumber p_3(u,t)& = \frac{e^{-\lambda t}}{c}\bigg[\lambda I_0\left(\frac{\lambda}{c}\sqrt{c^2t^2-u^2}\right)+\frac{\partial}{\partial t}I_0\left(\frac{\lambda}{c}\sqrt{c^2t^2-u^2}\right)+\frac{2c^2}{\lambda}\frac{\partial^2}{\partial u^2}I_0\left(\frac{\lambda}{c}\sqrt{c^2t^2-u^2}\right)+\\
\nonumber &+\frac{4c^2}{\lambda^2}\frac{\partial^3}{\partial t\partial u^2}I_0\left(\frac{\lambda}{c}\sqrt{c^2t^2-u^2}\right)\bigg] =\\ \nonumber&=p_2(u,t)+\frac{4c}{\lambda^2}e^{-\lambda t}\frac{\partial^3}{\partial t\partial u^2}I_0\left(\frac{\lambda}{c}\sqrt{c^2t^2-u^2}\right)
\end{align}
entails that we must concentrate our attention on the last term $$\frac{4c}{\lambda^2}e^{-\lambda t}\frac{\partial^3}{\partial t\partial u^2}I_0\left(\frac{\lambda}{c}\sqrt{c^2t^2-u^2}\right).$$
We observe that 
\begin{equation}\nonumber
\frac{\partial^3}{\partial t\partial u^2}I_0\left(\frac{\lambda}{c}\sqrt{c^2t^2-u^2}\right)=
-2^2 c^2 t\sum_{k=2}^\infty \left(\frac{\lambda}{2c}\right)^{2k}\frac{(c^2t^2-u^2)^{k-2}}{(k-2)!k!}+
4u^2 2c^2t \sum_{k=3}^\infty \left(\frac{\lambda}{2c}\right)^{2k}\frac{(c^2t^2-u^2)^{k-3}}{(k-3)!k!}
\end{equation}
or
\begin{align}\label{tre0}
e^{-\lambda t}\frac{4c^2}{c\lambda^2}\frac{\partial^3}{\partial t\partial u^2}I_0\left(\frac{\lambda}{c}\sqrt{c^2t^2-u^2}\right)&=
-\frac{2^4 c^3 t\cdot e^{-\lambda t}}{\lambda^2}\sum_{k=2}^\infty \left(\frac{\lambda}{2c}\right)^{2k}\frac{(c^2t^2-u^2)^{k-2}}{(k-2)!k!}+\\
\nonumber &+\frac{u^2 2^5 c^3 t\cdot e^{-\lambda t}}{\lambda^2} \sum_{k=3}^\infty \left(\frac{\lambda}{2c}\right)^{2k}\frac{(c^2t^2-u^2)^{k-3}}{(k-3)!k!}
\end{align}
The term $k=0$ of (3.22) and the term $k = 2$ in \eqref{tre0} cancel each other out.
We now put together (3.22) (without $k=0$) and \eqref{tre0} (without $k=2$) thus obtaining
\begin{align}
\nonumber &e^{-\lambda t} \sum_{k=1}^\infty \frac{(\lambda t)^{2k+2}}{(2k+2)!}\frac{(2k+2)!}{k!(k+1)!}
\frac{(c^2t^2-u^2)^k}{(2ct)^{2k+1}}-e^{-\lambda t}\frac{2^4c^3t}{\lambda^2} \sum_{k=3}^\infty \left(\frac{\lambda}{2c}\right)^{2k}\frac{(c^2t^2-u^2)^{k-2}}{k!(k-2)!}+\\
\nonumber &+e^{-\lambda t}\frac{2^5c^3t u}{\lambda^2} \sum_{k=3}^\infty \left(\frac{\lambda}{2c}\right)^{2k}\frac{(c^2t^2-u^2)^{k-3}}{k!(k-3)!}=\\
\nonumber & = e^{-\lambda t} \sum_{k=1}^\infty \frac{(\lambda t)^{2k+2}}{(2k+2)!}\frac{(2k+2)!}{k!(k+1)!}
\frac{(c^2t^2-u^2)^k}{(2ct)^{2k+1}}+e^{-\lambda t}2^4c^3t \sum_{k=3}^\infty \frac{(\lambda t)^{2k-2}}{(2c)^{2k}}\frac{(2k-2)!}{(2k-2)!}\frac{(c^2t^2-u^2)^{k-2}}{k!(k-2)!t^{2k-2}}+\\
\nonumber &+e^{-\lambda t}2^5c^3t u^2 \sum_{k=3}^\infty \frac{(\lambda t)^{2k-2}}{(2c)^{2k}}\frac{(2k-2)!}{(2k-2)!}\frac{(c^2t^2-u^2)^{k-3}}{k!(k-3)!t^{2k-2}}=\\
\nonumber &= \sum_{k=1}^\infty \Pr\{N(t)=2k+2\}\frac{(2k+2)!}{k!(k+1)!}
\frac{(c^2t^2-u^2)^k}{(2ct)^{2k+1}}- \sum_{k=1}^\infty \Pr\{N(t)=2k+2\}\frac{(2k+2)!}{2^{2k}(ct)^{2k+1}}\frac{(c^2t^2-u^2)^k}{k!(k+1)!}+\\
\nonumber &+ \sum_{k=1}^\infty \Pr\{N(t)=2k+2\}\frac{(2k+2)!}{(ct)^{2k+1}}\frac{2u^2(c^2t^2-u^2)^{k-1}}{(k+2)!(k-1)!} = \\
\nonumber & = \sum_{k=1}^\infty \Pr\{N(t)=2k+2\}\frac{(2k+2)!}{(k+1)!(k-1)!}\frac{(c^2t^2-u^2)^{k-1}}{2^{2k}(ct)^{2k+1}}\bigg[(c^2t^2-u^2)\left[\frac{1}{2k}-\frac{1}{k(k+2)}\right]+\frac{2u^2}{k+2}\bigg]=\\
\nonumber & =\sum_{k=1}^\infty \Pr\{N(t)=2k+2\}\frac{(2k+2)!}{(k+2)!(k-1)!}\frac{(c^2t^2-u^2)^{k-1}(c^2t^2+3u^2)}{(2ct)^{2k+1}}
\end{align}
and this concludes the proof.
\end{proof}

\begin{os}
In the first cycle, the moving particle can reach each of the vertices of the 
octahedron $S^3_{ct}$ with uniformly distributed probability $e^{-\lambda t}/6$, the edges of $S^3_{ct}$ with probability $\lambda t e^{-\lambda t}$ and the faces of $S^3_{ct}$ with probability 
$\frac{(\lambda t)^2}{2} e^{-\lambda t}$ uniformly spread over them. If the number $k$ of changes of direction
is equal to or greater than 3 (and less then or equal to 5) we have that the position of the particle is uniformly distributed on the octahedrons $S_{\mathcal{U}^3}$ homothetic to $S^3_{ct}$ (see Fig.3), where 
$\mathcal{U}(t)$ is a random variable with the following conditional distributions
\begin{equation}
\begin{cases}
&\Pr\{\mathcal{U}(t)\in du|N(t) = 3\} = \frac{3!(c^2t^2+u^2)}{2^3(ct)^3}du\\
&\Pr\{\mathcal{U}(t)\in du|N(t) = 4\} = \frac{(c^2t^2+3u^2)}{2(ct)^3}du\\
&\Pr\{\mathcal{U}(t)\in du|N(t) = 5\} = \frac{5(c^4t^4-u^4)}{4(ct)^5}du
\end{cases}
\end{equation}  
 for $0<u<ct$.\\
 Furthermore 
 \begin{equation}
\begin{cases}
&\mathbb{E}\{\mathcal{U}(t)|N(t) = 3\} = \frac{9}{16}ct\\
&\mathbb{E}\{\mathcal{U}(t)|N(t) = 4\} = \frac{5}{8}ct\\
&\mathbb{E}\{\mathcal{U}(t)|N(t) = 5\} = \frac{5}{12}ct.
\end{cases}
\end{equation}  

\end{os}

\begin{os}
The general mean values of the conditional r.v. $\mathcal{U}(t)$ read
\begin{equation}
\begin{cases}
&\mathbb{E}\{\mathcal{U}(t)|N(t) = 2k+1\} = \frac{(2k+1)!(k+2)}{2^{2k+1}(k+1)!^2}ct\\
&\mathbb{E}\{\mathcal{U}(t)|N(t) = 2k+2\} = \frac{(2k+1)!(k+4)}{2^{2k+1}k!(k+2)!}ct
\end{cases}
\end{equation}
The mean values above can be rewritten also as 
\begin{equation}
\begin{cases}
&\mathbb{E}\{\mathcal{U}(t)|N(t) = 2k+1\} = \binom{2k+1}{k+1}\frac{ct}{2^{2k+1}}\frac{k+2}{k+1}\\
&\mathbb{E}\{\mathcal{U}(t)|N(t) = 2k+2\} = \binom{2k+1}{k+1}\frac{ct}{2^{2k+1}}\frac{k+4}{k+2}
\end{cases}
\end{equation}
and thus
\begin{equation}
\frac{\mathbb{E}\{\mathcal{U}(t)|N(t) = 2k+1\}}{\mathbb{E}\{\mathcal{U}(t)|N(t) = 2k+2\}} = 
\frac{(k+2)^2}{(k+1))(k+4)} = 1-\frac{3k}{(k+1)(k+4)}
\end{equation}
and for $k\rightarrow \infty$ they will coincide. The formulas above can be written in terms
of Catalan numbers
$$C_k = \binom{2k}{k}\frac{1}{k+1}, \quad k\geq 1.$$
\end{os}

\begin{os}
We observe that the one-dimensional counterpart of the cyclic motions treated here is
the telegraph process $T(t)$ and its folded conditional distributions are
\begin{equation}\label{aa}
\begin{cases}
& \Pr\{|T(t)|\in du|N(t)= 2k+2\}= \frac{(2k+2)!}{k!(k+1)!}\frac{(c^2t^2-u^2)^k}{(2ct)^{2k+1}}du, \quad k\geq 0,\\
&\Pr\{|T(t)|\in du|N(t)= 2k+1\}= \frac{(2k+1)!}{k!^2}\frac{(c^2t^2-u^2)^k}{2^{2k}(ct)^{2k+1}}du, \quad k\geq 0.
\end{cases}
\end{equation}
In the planar cyclic motion of Section 3, the even-order probabilities coincide with those of the first line of (4.16) (distributed however on squares homothetic to $S_{ct}$).
In $\mathbb{R}^3$ the odd-order distributions coincide with the odd-order conditional density of the planar cyclic motions. Also in this case, the distributions are uniform on the 
octahedrons homothetic to $S^3_{ct}$. We have therefore that (we write $\mathcal{U}_1(t)$ for $|T(t)|$)
\begin{align}
\nonumber &\Pr\{\mathcal{U}_1(t)\in du|N(t)= 2k+2\}= \Pr\{\mathcal{U}_2(t)\in du |N(t)= 2k+2\}\\
\nonumber &\Pr\{\mathcal{U}_2(t)\in du|N(t)= 2k+1\}= \Pr\{\mathcal{U}_3(t)\in du |N(t)= 2k+1\}.
\end{align}
We conjecture therefore that for motions in Euclidean spaces of arbitrary dimension $d$, we must have
\begin{align}
\nonumber &\Pr\{\mathcal{U}_{2d-1}(t)\in du|N(t)= 2k+2\}= \Pr\{\mathcal{U}_{2d}(t)\in du |N(t)= 2k+2\}\\
\nonumber &\Pr\{\mathcal{U}_{2d}(t)\in du|N(t)= 2k+1\}= \Pr\{\mathcal{U}_{2d+1}(t)\in du |N(t)= 2k+1\}.
\end{align}
\end{os}

Observe that similar results can be derived for general higher order dimensions $d\geq 2$.
Also in this case, after an exponential change of variable, i.e. $p(x_1, \dots, x_d,t) =
e^{-\lambda t} w(x_1, \dots, x_d, t)$, the governing equation has the form 
\begin{equation}
\prod_{i=1}^d \left(\frac{\partial^2}{\partial t^2}-c^2\frac{\partial^2}{\partial x_i^2}\right)w-
\lambda^{2d}w = 0
\end{equation}
and the general structure of the solution (after a change of variable in $(u,t)$) is given by
\begin{equation}
p(u,t) = e^{-\lambda t}\sum_{i=0}^d A_i\frac{\partial^i}{\partial t^i}I_0\left(\frac{\lambda}{c}\sqrt{c^2t^2-u^2}\right),
\end{equation}
where $A_i$ are suitably chosen coefficients, such that 
\begin{equation}
\int_0^{ct} p(u,t)du = 1-\sum_{k=0}^{d-1} \frac{(\lambda t)^k}{k!} e^{-\lambda t}.
\end{equation}

\begin{os}

The equation 
\begin{equation}
\prod_{k=1}^2 \bigg[\left(\frac{\partial}{\partial t}+\lambda\right)^2-c^2\frac{\partial^2}{\partial x_k^2}\bigg]p-
\lambda^{4}p = 0
\end{equation}
converges to the heat equation 
\begin{equation}
\frac{\partial p}{\partial t} = \frac{1}{4}\left(\frac{\partial^2}{\partial x_1^2}+
\frac{\partial^2}{\partial x_2^2}\right)p
\end{equation}
as $\lambda, c\rightarrow +\infty$ in such a way that $c^2/\lambda \rightarrow 1$.\\
The general equation 
\begin{equation}\label{dee}
\prod_{k=1}^d \bigg[\left(\frac{\partial}{\partial t}+\lambda\right)^2-c^2\frac{\partial^2}{\partial x_k^2}\bigg]p-
\lambda^{2d}p = 0
\end{equation}
converges to the heat equation
\begin{equation}\label{lass}
\frac{\partial p}{\partial t} = \frac{1}{2d}\Delta p,
\end{equation}
where the \textit{volatility} $\sigma^2 = 1/d$.
The crucial terms in developing the operator \eqref{dee} are
  $$\left(\frac{\partial}{\partial t}+\lambda\right)^{2d}$$
and
\begin{equation}\label{las}
-c^2\left(\frac{\partial}{\partial t}+\lambda\right)^{2d-2}\sum_{k=1}^d\frac{\partial^2}{\partial x_k^2}.
\end{equation}
By dividing \eqref{las} by $\lambda^{2d-1}$, \eqref{las} can be rewritten as
\begin{equation}
\nonumber -\frac{c^2}{\lambda}\left(\frac{1}{\lambda}\frac{\partial}{\partial t}+1\right)^{2d-2}\sum_{k=1}^d\frac{\partial^2}{\partial x_k^2}
\end{equation}
and in the limit yields the right-hand side of \eqref{lass}.
In 
\begin{equation}\label{ultimas}
\left(\frac{\partial}{\partial t}+\lambda\right)^{2d} = \sum_{k=0}^{2d} \binom{2d}{k}\lambda^k
\left(\frac{\partial}{\partial t}\right)^{2d-k}.
\end{equation}
The term $k=0$ in \eqref{ultimas} cancels $\lambda^{2d}$ of \eqref{dee} and $k = 2d-2$ reads
$$(2d)\lambda^{2d-1}\left(\frac{\partial}{\partial t}\right)$$
and thus dividing by $\lambda^{2d-1}$ yields the left-hand side of the heat equation \eqref{lass}.
\end{os}

 \bigskip
                  
                  \textbf{Acknowledgment:} The authors thank Dr. F. Iafrate for his help on 
                  the graphical simulations presented in the paper.

\end{document}